\documentclass[a4paper,12pt]{amsart}

\usepackage{mathrsfs}

\usepackage{amssymb}
\usepackage{latexsym}
\usepackage{amsfonts}
\usepackage{amsmath}
\usepackage{eucal}
\usepackage{bm}
\usepackage{bbm}
\usepackage{graphicx}
\usepackage[english]{varioref}
\usepackage[nice]{nicefrac}
\usepackage[all]{xy}
\usepackage{amsthm}
\usepackage[margin=1.25in]{geometry}
\usepackage{amsmath,amscd}

\newcommand{\Id}{\textrm{id}}

\def\i{^{-1}}
\def\ge{\geqslant}
\def\le{\leqslant}
\def\<{\langle}
\def\>{\rangle}

\def\bc{\bold{c}}
\def\h{\text{ht}}

\def\dom{\text{dom}}

\def\a{\alpha}
\def\b{\beta}
\def\g{\gamma}

\def\d{\delta}
\def\D{\Delta}
\def\L{\Lambda}
\def\e{\epsilon}

\def\o{\omega}

\def\s{\sigma}
\def\t{\tau}

\def\k{\kappa}
\def\l{\lambda}

\def\ZZ{\mathbb Z}
\def\AA{\mathbb A}

\def\QQ{\mathbb Q}
\def\JJ{\mathbb J}
\def\FF{\mathbb F}

\def\kk{\bold{k}}

\def\ca{\mathcal A}

\def\cc{\mathcal C}

\def\cl{\mathcal L}

\def\co{\mathcal O}

\def\car{\mathcal R}

\def\cz{\mathcal Z}

\def\car{\mathcal R}

\def\fs{\mathfrak S}

\def\bc{\bold{c}}

\def\dft{{\rm def}}
\def\Irr{{\rm Irr}}

\def\h{{\rm h}}
\def\tp{{\rm top}}

\def\GL{{\rm GL}}
\def\GSp{{\rm GSp}}
\def\Res{{\rm Res}}
\def\inv{{\rm inv}}
\def\dft{{\rm def}}
\def\GU{{\rm GU}}

\theoremstyle{plain}
\newtheorem{thm}{Theorem}[section]
\newtheorem{conj}{Conjecture}[section]

\newtheorem*{thm*}{Theorem}
 \newtheorem{prop}[thm]{Proposition}
 \newtheorem{lem}[thm]{Lemma}
 \newtheorem{cor}[thm]{Corollary}

\theoremstyle{definition}

\theoremstyle{remark}

\newtheorem*{rmk}{Remark}
\newtheorem*{claim*}{Claim}

\begin{document}
\author{Sian Nie}
\address{Institute of Mathematics, Academy of Mathematics and Systems Science, Chinese Academy of Sciences, 100190, Beijing, China}
\email{niesian@amss.ac.cn}

\thanks{This work is supported by the National Natural Science Foundation of China (No. 11621061 and No. 11501547) and by Key Research Program of Frontier Sciences, CAS, Grant No. QYZDB-SSW-SYS007.}

\title{Semi-modules and irreducible components of affine Deligne-Lusztig varieties}
\begin{abstract}
Let $G$ be the Weil restriction of a general linear group. By extending the method of semi-modules developed by de Jong, Oort, Viehmann and Hamacher, we obtain a stratification of the affine Deligne-Lusztig varieties for $G$ (in the affine Grassmannian) attached to a minuscule coweight and a basic element. As an application, we verify a conjecture by Chen and Zhu on irreducible components of affine Deligne-Lusztig varieties for $G$.
\end{abstract}

\maketitle

\section*{Introduction}
Affine Deligne-Lusztig varieties are closely related to the Rapoport-Zink moduli spaces of $p$-divisible groups (cf. \cite{R}, \cite{RZ}), and play an important role in the study of Shimura varieties. There has been an extensive study on affine Deligne-Lusztig varieties. However, many basic aspects of their geometric structure are not fully understood yet. We refer to \cite{He} and \cite{HV} for the current status of these topics.

In this note, we study the affine Deligne-Lusztig varieties $X_\mu^G(\g)$ in affine Grassmannians, where $G$ is the Weil restriction of a general linear group. By extending the method of semi-modules (or extended EL-charts) developed by de Jong-Oort \cite{dJO}, Viehmann \cite{V} and Hamacher \cite{H}, we show that if $\g$ is basic and $\mu$ is minuscule, there is a stratification (in the loose sense) of $X_\mu^G(\g)$ parameterized by semi-modules. As an application, we verify a conjecture by Miaofen Chen and Xinwen Zhu concerned with the irreducible connected components of  $X_\mu^G(\g)$.

\

To describe the results more precisely, we introduce some notation. Let $\FF_q$ be a finite field with $q$ elements, and let $\kk$ be an algebraic closure of $\FF_q$. Denote by $F=\FF_q((t))$ and $L=\kk((t))$ the fields of Laurent series, whose integer rings are denoted by $\co_F=\FF[[t]]$ and $\co=\kk[[t]]$ respectively. Let $\s$ denote the Frobenius automorphism of $L / F$.

Let $G$ be a connected reductive group over $\FF_q$. Fix $S \subseteq T \subseteq B \subseteq G$, where $S$ is a maximal split torus, $T$ a maximal torus and $B$ a Borel subgroup of $G$. Denote by $X_*(T)$ for the cocharacter group of $T$, and by $X_*(T)_{G-\dom}$ the set of dominant cocharacters determined by $B$. Let $\leq$ denote the dominance partial order on $X_*(T)$ defined by $B$. Let $\le$ denote the partial order on $X_*(T)_\QQ$ such that $v \le v'$ if and only if $v' - v$ is a non-negative linear combination of positive coroots.

We have the Cartan decomposition $G(L) = \sqcup_{\l \in X_*(T)_{G-\dom}} K t^\l K$, where $K=G(\co)$. For $\g \in G(L)$ and $\mu \in X_*(T)$, the attached affine Deligne-Lusztig variety is defined by $$X_\mu^G(\g)=\{g K \in G(L)/K; g\i \g \s(g) \in K t^\mu K\}.$$ It carries a natural action by the group $\JJ_\g^G=\{g \in G(L); g\i \g \s(g)=\g\}$. By definition, $X_\mu^G(\g)$ only depends on the $\s$-conjugacy class $[\g]_G$ of $\g$. Thanks to Kottwitz \cite{Ko1}, $[\g]_G$ is uniquely determined by two invariants: the Newton point $\nu_G(\g) \in X_*(S)_{\QQ, G-\dom}$ and the Kottwitz point $\k_G(\g) \in \pi_1(G)_\s$; see \cite[\S 2.1]{HV}. By \cite{KR} and \cite{Ga}, $X_\mu^G(\g) \neq \emptyset$ if and only if $\k_G(t^\mu)=\k_G(\g)$ and $\nu_G(\g) \le \mu^\diamond$, where $\mu^\diamond$ denotes the $\s$-average of $\mu$. Moreover, thanks to \cite{GHKR}, \cite{V}, \cite{H}, \cite{Z} and \cite{HaV}, $X_\mu^G(\g)$ is locally of finte type, equi-dimensional and $$\dim X_\mu^G(\g)=\<\rho_G, \mu-\nu_G(\g)\> - \frac{1}{2} \dft_G(\g),$$ where $\rho_G$ is the half-sum of positive roots of $G$ and $ \dft_G(\g)$ denotes the {\it defect} of $\g$; see \cite[\S 1.9.1]{Ko2}.

The stratification by semi-modules was first considered by de Jong and Oort \cite{dJO} for $X_\mu^G(\g)$ with $G$ split, $\mu$ minuscule and $\g$ superbasic. It was latter extended by Viehmann \cite{V} and Hamacher \cite{H} to the case where $\g$ is superbasic. Recently, Chen and Viehmann \cite{CV} defined a stratification for all affine Deligne-Lusztig varieties, which recovers the stratification by semi-modules in the superbasic case.
\begin{thm} (=Corollary \ref{dec})
Suppose $G=\Res_{\FF_{q^d}/\FF_q} \GL_n$. If $\mu$ is minuscule and $\g$ is basic (i.e. $\nu_G(\g)$ is central for $G$), then there is a decomposition $$X_\mu^G(\g)=\sqcup_A \cl(A),$$ where $A$ ranges over all semi-modules of Hodge type $\mu$ (see \S\ref{setup}) and $\cl(A) \subseteq X_\mu^G(\g)$ is a locally closed subset, which is finite and smooth over an affine space.
\end{thm}

\

Now we turn to the study of irreducible components of $X_\mu^G(\g)$. In \cite[\S 2.1]{HV}, Hamacher and Viehmann defined the ``best integral approximation" $\l_G(\g)$ of the Newton point $\nu_G(\g)$, which is the unique maximal element in the set $$\{\l \in X_*(T)_\s; \k_G(t^\l)=\k_G(\g), \l^\diamond \le \nu_G(\g)\}.$$ Moreover, if $M \supseteq T$ is a standard Levi subgroup such that $\g \in M(L)$ and $\nu_M(\g)=\nu_G(\g)$, then $\l_M(\g)=\l_G(\g)$.

Let $\hat S \subseteq \hat T \subseteq \hat B \subseteq \hat G$ be the dual of $S \subseteq T \subseteq B \subseteq G$ in the sense of Deligne and Lusztig. We have canonical identifications $X_*(T)=X^*(\hat T)$, $X_*(T)_{G-\dom}=X^*(\hat T)_{\hat G-\dom}$ and so on. For $\mu \in X_*(T)_{G-\dom}$ let $V_\mu^{\hat G}$ denote the irreducible $\hat G$-module with highest weight $\mu$. Moreover, for $\l \in X^*(\hat S)$, the $\l$-weight space of $V_\mu^{\hat G}$ is denoted by $V_\mu^{\hat G}(\l)$.

\begin{conj} [Chen, Zhu] \label{conj}
Let notations be as above. Then there exists a natural bijection between $\JJ_\g^G \backslash \Irr X_\mu^G(\g)$ and a basis of $V_\mu^{\hat G}(\l_G(\g))$ related to the Mirkovic-Vilonen cycles (see \cite{MV}). In particular, $$|\JJ_\g^G \backslash \Irr X_\mu^G(\g)| = \dim V_\mu^{\hat G}(\l_G(\g)).$$ Here $\Irr X_\mu^G(\g)$ denotes the set of irreducible components of $X_\mu^G(\g)$.
\end{conj}
If $\mu$ is minuscule and $\g$ is superbasic, the conjecture is proved by Hamacher and Viehmann \cite{HV}. If $\g$ is {\it unramified}, that is, $\dft_G(\g)=0$, it is proved by Xiao and Zhu \cite{XZ}. In both cases, the authors obtained a complete description of $\Irr X_\mu^G(\g)$.

\begin{rmk}
 We mention that a complete description of $\Irr X_\mu^G(\g)$ was also known for the case where $G$ equals $\GL_n$ or $\GSp_{2n}$ and $\mu$ is minuscule (cf. \cite{V1}, \cite{V2}), and for the case where $(G, \mu)$ is {\it fully Hodge-Newton decomposable}; see \cite{VW} for a typical example in this case.
\end{rmk}

\begin{thm} (=Corollary \ref{main''}) \label{main}
Conjecture \ref{conj} holds for $G=\Res_{\FF_{q^d} / \FF_q} \GL_n$.
\end{thm}
To prove the theorem, we first use geometric Satake to reduce the problem to the case where $\mu$ is minuscule and $\g$ is basic. Then we can use Theorem 1.1 to show that the $\JJ_\g^G$-orbits of $\Irr X_\mu^G(\g)$ are parameterized by equivalence classes of {\it rigid semi-modules} $A$ (see \S\ref{rigid}) such that $\dim \cl(A) = \dim X_\mu^G(\g)$. Finally, we show the number of such rigid semi-modules coincides with the dimension of $V_\mu^{\hat G}(\l_G(\g))$. This is accomplished by a reduction to the superbasic case, which has been solved by Hamacher and Viehmann \cite[Theorem 1.5]{HV}.

\subsection*{Acknowledgements} The author is grateful to Viehmann for detailed explanations on her joint work \cite{HV} with Hamacher, which plays an essential role in this note. He also would like to thank Chen and G\"{o}rtz for helpful comments.

\section{Reduction to the minuscule case} \label{reduce}
In this section, we reduce Theorem \ref{main} to the minuscule case. Let $G$ be a split reductive group over $\FF_q$. Let $d \in \ZZ_{\ge 1}$ and let $H=H_d$ be a reductive group over $\kk$ such that $$H=\prod_{\t \in \ZZ_d} H_\t,$$ where $\ZZ_d=\ZZ / d\ZZ$ for some $d \in \ZZ_{\ge 1}$ and $H_\t = G \otimes_{\FF_q} \kk$ for $\t \in \ZZ_d$. Let $\s^H$ be an automorphism of $H(L)$ defined by $(g_1, g_2, \dots, g_d) \mapsto (g_2, \dots, g_r, \s(g_1))$, where $\s$ denotes the Frobenius automorphism $\s$ on $G(L)$.

Let $T$ be a split maximal torus of $G$. Then $T_H :=\prod_\t T$ is a maximal torus of $H$. Via diagonal embedding, we can identify $T$ with the maximal $\s^H$-split torus $S_H$ in $T_H$. The embedding $\l \mapsto \l^H=(0, \dots, 0, \l)$ identifies $X_*(T)_\s=X_*(T)$ with $X_*(T_H)_{\s^H} = X^*(\hat S_H)$. Similarly, the embedding $\g \mapsto \g^H=(1, \dots, 1, \g)$ gives a bijection between $\s$-conjugacy classes of $G(L)$ and $\s^H$-conjugacy classes of $H(L)$. By abuse of notation, we will identify $\g$ (resp. $\l$) with $\g^H$ (resp. $\l^H$).
\begin{lem} \label{approximation}
Let $\g \in G(L)$. Then $\l_H(\g)=\l_G(\g) \in X_*(T) = X^*(\hat S_H)$.
\end{lem}

Let $\mu_\bullet=(\mu_1, \dots, \mu_d) \in X_*(T_H)_{H-\dom} = \prod_\t X_*(T)_{G-\dom}$. We consider the twisted product $$\cz_{\mu_\bullet}=K t^{\mu_1} K \times_K \cdots \times_K K t^{\mu_d} K / K$$ together with the convolution map $$m_{\mu_\bullet}: \cz_{\mu_{\bullet}} \to \overline{K t^{|\mu_\bullet|} K / K} = \cup_{\l \leq |\mu_\bullet|} K t^\l K / K$$ given by $(g_1, \dots, g_{d-1}, g_d K) \mapsto g_1 \cdots g_d K$, where $|\mu_\bullet|=\mu_1 + \cdots + \mu_d$.

Consider the following decomposition of $\hat G$-modules $$V_{\mu_\bullet}^{\hat H}=V_{\mu_1}^{\hat G} \otimes \cdots \otimes V_{\mu_d}^{\hat G} = \bigoplus_{\chi \in X_*(T)_{G-\dom}} (V_\chi^{\hat G})^{\oplus a_{\mu_{\bullet}}^\chi},$$ where $a_{\mu_\bullet}^\chi$ is the multiplicity of $V_\chi^{\hat G}$ in $V_{\mu_\bullet}^{\hat H}$ (as $\hat G$-modules).
\begin{thm} [\cite{MV}, \cite{NP}, \cite{Haines}] \label{sat}
Suppose $\mu_\bullet$ is a sum of dominant minuscule cocharacters. For each $y \in K t^\l K / K$ with $\l \leq |\mu_\bullet|$ dominant, the fiber $m_{\mu_\bullet}\i(y)$ is equi-dimensional of dimension $\<\rho_G, |\mu_\bullet|-\l\>$ if $m_{\mu_\bullet}\i(y) \neq \emptyset$. Moreover, the number of irreducible components of $m_{\mu_\bullet}\i(y)$ equals $a_{\mu_\bullet}^\l$.
\end{thm}

For $\g \in G(L)$ and $\mu_\bullet \in X_*(T_H)_{H-\dom}$, we define $$X_{\mu_\bullet}^H(\g)=\{h K_H \in H(L)/K_H; h\i \g \s^H(h) \in K_H t^{\mu_\bullet} K_H\},$$ where $K_H=\prod_\t K$ and $K=G(\co)$. Thanks to Zhu \cite[\S 3.13]{Z}, there is a Catesian square \begin{align*} \xymatrix{
  X_{\mu_\bullet}^H(\g) \ar[d]^\a \ar[r] & G(L) \times_K \cz_{\mu_\bullet}  \ar[d]^{\Id \times_K m_{\mu_\bullet}} \\
\cup_{\l \leq |\mu_\bullet|} X_\l^G(\g) \ar[r] & G(L) \times_K \overline{K t^{|\mu_\bullet|} K / K},  }\end{align*} where the bottom horizontal map is given by $g K \mapsto (g, g\i \g \s(g) K)$, and the top horizontal map is given by $(g_1K, \dots, g_d K) \mapsto (g_1, g_1\i g_2, \dots, g_{d-1}\i g_d, g_d\i \g \s(g_1) K)$. Moreover, via the identification $$\JJ_\g^G \cong \JJ_{\g}^H := \{h \in H(L); h\i \g \s^H(h) = \g\}, \quad g \mapsto (g, \dots, g),$$ the above Catesian square is $\JJ_\g^G$-equivariant (by left multiplication).
\begin{lem} \label{compare}
Let notations be as above. Suppose $\mu_\bullet$ is sum of minuscule dominant coweights. Then $$\Irr X_{\mu_\bullet}^H(\g) = \sqcup_{\l \leq |\mu_\bullet|, \ a_{\mu_\bullet}^\l \neq 0} \sqcup_{\bc \in \Irr X_\l^G(\g)} \Irr (\a\i(\bc)).$$ In particular, $$\JJ_{\g}^H \backslash \Irr X_{\mu_\bullet}^H(\g) = \sqcup_{\l \leq |\mu_\bullet|, \ a_{\mu_\bullet}^\l \neq 0} \sqcup_{\bc \in \JJ_\g^G \backslash \Irr X_\l^G(\g)} \Irr (\a\i(\bc))$$ and hence $$|\JJ_{\g}^H \backslash \Irr X_{\mu_\bullet}^H(\g)| = \sum_{\l \leq |\mu_\bullet|} a_{\mu_\bullet}^\l |\JJ_\g^G \backslash \Irr X_\l^G(\g)|.$$ Here $\l$ always denotes a dominant cocharacter.
\end{lem}
\begin{proof}
Let $\bc \in \Irr X_\l^G(\g)$. If $\a\i(X_\l^G(\g)) \neq \emptyset$, that is, $m_{\mu_\bullet}\i (t^\l K) \neq \emptyset$, or equivalently, $a_{\mu_\bullet}^\l \neq 0$, then Theorem \ref{sat} tells that $\a\i(\bc)$ is equi-dimensional and $$\dim \a\i(\bc) = \dim \bc + \<\rho_G, |\mu_\bullet|-\l\> = \dim  X_\l^G(\g) + \<\rho_G, |\mu_\bullet|-\l\> = \dim X_{\mu_\bullet}^H(\g).$$ The proof is finished.
\end{proof}

\begin{prop} \label{minu-basic}
Suppose $G=\GL_n$. If $\g$ is basic and $\mu_\bullet \in X_*(T_H)_{H-\dom}$ is minuscule. Then $|\JJ_{\g}^H \backslash \Irr X_{\mu_\bullet}^H(\g)| =  \dim V_{\mu_\bullet}^{\hat H}(\l_H(\g))$.
\end{prop}
The proof is given in \S \ref{irr}; see Corollary \ref{main''}.

\begin{prop} \label{minu}
Suppose $G=\GL_n$ and $\mu_\bullet \in X_*(T_H)_{H-\dom}$ is minuscule. Then $$|\JJ_{\g}^H \backslash \Irr X_{\mu_\bullet}^H(\g)| = \dim V_{\mu_\bullet}^{\hat H}(\l_H(\g)).$$
\end{prop}
\begin{proof}
Let $T_H \subseteq M_\g \subseteq P_\g \subseteq H$ be the standard Levi and parabolic subgroups associated to $\nu_H(\g)$. Let $X_*(T_H)_{M_\g-\dom}^{\leq \mu_\bullet}$ denote the set of $M_\g$-dominant cocharactors $\chi_\bullet$ such that $\chi_\bullet \leq \mu_\bullet$. Let $I_{\mu_\bullet, \g} = I_{\mu_\bullet, \g, M_\g}$ be the set of cocharacters $\chi_\bullet \in X_*(T_H)_{M_\g-\dom}^{\leq \mu_\bullet}$ such that $\k_{M_\g}(\chi_\bullet)=\k_{M_\g}(\g)$. By definition, $\JJ_\g^H \subseteq M_\g(L)$. Then \cite[Corollary 5.9]{HV} tells that there is a natural bijection $$\JJ_{\g}^H \backslash \Irr X_{\mu_\bullet}^H(\g) \overset \sim \longrightarrow \bigsqcup_{\chi_\bullet \in I_{\mu_\bullet, \g}} \JJ_{\g}^H \backslash \Irr X_{\chi_\bullet}^{M_\g}(\g),$$ where $ I_{\mu_\bullet, \g}$ is defined analogously as in \cite[\S 5]{HV}.
Therefore, \begin{align*} |\JJ_{\g}^H \backslash \Irr X_{\mu_\bullet}^H(\g)| &= \sum_{\chi_\bullet \in I_{\mu_\bullet, \g}} |\JJ_{\g}^H \backslash X_{\chi_\bullet}^{M_\g}(\g)| \\  &= \sum_{\chi_\bullet \in I_{\mu_\bullet, \g}} \dim V_{\chi_\bullet}^{\hat M_\g}(\l_{M_\g}(\g)) \\ &=  \sum_{\chi_\bullet \in X_*(T_H)_{M_\g-\dom}^{\leq \mu_\bullet}} \dim V_{\chi_\bullet}^{\hat M_\g}(\l_{M_\g}(\g)) \\ &= \dim V_{\mu_\bullet}^{\hat H}(\l_{M_\g}(\g)) \\ &= \dim V_{\mu_\bullet}^{\hat H}(\l_H(\g)),  \end{align*} where the second equality follows from Proposition \ref{minu-basic} as $\g$ is basic in $M_\g(L)$; the third one follows from that $V_{\chi_\bullet}^{\hat M_\g}(\l_{M_\g}(\g)) = 0$ (for $\chi_\bullet \in X_*(T_H)_{M_\g-\dom}^{\leq \mu_\bullet}$) unless $\chi_\bullet \in I_{\mu_\bullet, \g}$; the fourth one follows from that $V_{\mu_\bullet}^{\hat H} = \oplus_{\chi_\bullet \in X_*(T_H)_{M_\g-\dom}^{\leq \mu_\bullet}} V_{\chi_\bullet}^{\hat M_\g}$ since $\mu_\bullet$ is minuscule.
\end{proof}

\begin{cor} \label{gl}
Suppose $G=\GL_n$. Then $|\JJ_\g^G \backslash X_\mu^G(\g)| = \dim V_\mu^{\hat G}(\l_G(\g))$.
\end{cor}
\begin{proof}
If $\mu$ is minuscule, it is proved in \cite{HV}. Assume it is true for all dominant cocharacters $\mu'$ such that $\mu' < \mu$. We show it is also true for $\mu$. Since $G=\GL_n$, there exist $d' \in \ZZ_{\ge 1}$ and a $d'$-tuple $\mu_\bullet$ of dominant minuscule cocharacters (of $G$) such that $\mu=|\mu_\bullet|$. Set $H=H_{d'}$. Thus \begin{align*} \dim V_{\mu_\bullet}^{\hat H}(\l_G(\g)) &= |\JJ_{\g}^H \backslash \Irr X_{\mu_\bullet}^H(\g)| \\ &= \sum_{\chi \leq \mu} a_{\mu_\bullet}^\chi |\JJ_\g^G \backslash \Irr X_\chi^G(\g)| \\ &= |\JJ_\g^G \backslash \Irr X_\mu^G(\g)| + \sum_{\chi < \mu} a_{\mu_\bullet}^\chi |\JJ_\g^G \backslash \Irr X_\chi^G(\g)|  \\ & = |\JJ_\g^G \backslash \Irr X_\mu^G(\g)| + \sum_{\chi < \mu} a_{\mu_\bullet}^\chi \dim V_\chi^{\hat G}(\l_G(\g)), \end{align*} where the first equality follows from Proposition \ref{minu}; the second one follows from Lemma \ref{compare}; the last one follows from induction hypothesis. On the other hand, we have $\dim V_{\mu_\bullet}^{\hat H}(\l_H(\g))=\sum_{\chi \leq \mu} a_{\mu_\bullet}^\chi \dim V_\chi^{\hat G}(\l_G(\g))$ by Lemma \ref{approximation}. Therefore, $|\JJ_\g^G \backslash X_\mu^G(\g)| = \dim V_\mu^{\hat G}(\l_G(\g))$ as desired.
\end{proof}

Combining Lemma \ref{compare} with Corollary \ref{gl}, we deduce that
\begin{cor} \label{main'}
Suppose $G=GL_n$. Then $|\JJ_{\g}^H \backslash X_{\mu_\bullet}^H(\g)| = V_{\mu_\bullet}^{\hat H}(\l_H(\g))$.
\end{cor}

\section{Decomposition by semi-modules} \label{semi}
Let notations be as in \S\ref{reduce}, except that we use $\chi$, instead of $\chi_\bullet$, to denote a cocharacter in $X_*(T_H)$. Moreover, without loss of generality, we can assume $H=\Res_{\FF_{q^d} / \FF_q} \GL_n$ for simiplicity. We will use the method of semi-modules to give a decomposition for $X_\mu^H(\g)$, where $\mu$ is minuscule and $\g$ is basic.

\

For each $\t \in \ZZ_d$ let $N_\t=\oplus_{1 \le i \le n} L e_{\t, i}$. Then $H_\t(L) \cong \GL_n(N_\t)$ and $$H(L) = \prod_{\t \in \ZZ_d} H_\t(L) \cong \prod_{\t \in \ZZ_d} \GL_n (N_\t).$$ Via this identification, we can specify the notations in \S\ref{reduce} as follows.
\begin{itemize}
\item $\s=\s^H: H(L) \to H(L)$ is induced by $e_{\t, i} \mapsto e_{\t-1, i}$. \\

\item  $\g \in H(L)$ is given by $e_{\t, i} \mapsto e_{\t, i+m}$ for some $m \in \ZZ_{\ge 0}$ if $\t=0 \in \ZZ_d$ and $e_{\t, i} \mapsto e_{\t, i}$ otherwise. Here we adopt the convention that $e_{\t, i+n}=t e_{\t, i}$.\\

\item $S_H \subseteq T_H \subseteq B_H$ denote the split diagonal torus, the diagonal torus and the upper triangular Borel subgroup, respectively. \\

\item $X_*(T_H)=(\ZZ^n)^{\ZZ_d}$ and $\chi=(\chi_\t)_{\t \in \ZZ_d} \in X_*(T_H)$ is dominant if and only if each $\chi_\t$ is dominant, that is, $\chi_\t(1) \ge \cdots \ge \chi_\t(n)$. \\

\item $\mu=(\mu_\t)_{\t \in \ZZ_d}$ is the fixed dominant minuscule cocharacter. Let $m_\t=\sum_{i=1}^n \mu_\t(i)$ for $\t \in \ZZ_d$. As $X_\mu^H(\g) \neq \emptyset$, we have $m=\sum_\t m_\t$. \\

\item $K_H=\prod_\t GL_n(\D_\t)$, where $\D_\t = \oplus_{i=1}^n \co e_{\t, i}$ is the standard $\co$-lattice in $N_\t$. The map $g_\t \mapsto g_\t \D_\t$ induces a bijection $$H(L) / K_H \cong \{\L=(\L_\t)_{\t \in \ZZ_d}; \L_\t \text{ is a lattice in $N_\t$} \}.$$
\end{itemize}
Let $\L, \L' \in H(L)/K_H$. Denote by $\inv(\L, \L')$ the unique dominant cocharacter $\chi$ such that $g(\L, \L')=(\D, t^\chi \D)$ for some element $g \in G(L)$. Now the affine Deligne-Lusztig variety $X_\mu^H(\g)$ is given by $$X_\mu(\g)=X_\mu^H(\g)=\{\L \in H(L)/K_H; \inv(\L, \g \s(\L))=\mu\}.$$

\subsection{} \label{setup} Let $O=\ZZ_d \times \ZZ$. Let $(\t', i'), (\t, i) \in O$. We write $(\t, i) \le (\t', i')$ if $\t=\t'$ and $i \le i'$. For $k \in \ZZ$ we set $k+(\t, i)=(\t, i)+k=(\t, i+k) \in O$. For $a=(\t, i) \in O$ we set $e_a=e_{\t, i} \in N_\t$. For $v \in N_\t$ we define $\h(v)=\max_{\le} \{a; v \in \sum_{j=0}^\infty \kk e_{a+j}\}$.

Let $h \in \ZZ_{\ge 1}$ be the greatest common divisor of $m$ and $n$. Set $n'=n/h$ and $m'=m/h$. Let $\t \in \ZZ_d$ and $k \in \ZZ$. We set $O^k=\{(\t, j) \in O; j = k \mod h\}$ and $O_\t=\{(\t, j) \in O; j \in \ZZ\}$. For any subset $E \subseteq O$, we set $E^k=E \cap O^\e$, $E_\t=E \cap O_\t$ and $E_\t^k=E \cap O_\t \cap O^k$. Define $f: O \to O$ by $(\t, i) \mapsto (\t-1, i+m)$ if $\t=1$ and $(\t, i) \mapsto (\t-1, i)$ otherwise. Notice that $\g\s(e_a)=e_{f(a)}$ for $a \in O$.

We say a subset $A \subseteq O$ is a {\it semi-module} (for $H$) if $A$ is bounded below, $n+A, f(A) \subseteq A$ and $O=n\ZZ + A$. Set $\bar A = A \setminus (n+A)$. For $a \in A$, let $\varphi_A (a) \in \ZZ$ such that $f(a)- n\varphi_A (a) \in \bar A$. Then we have a bijection $r_A: \bar A \overset \sim \to \bar A$ by $b \mapsto f(b)- n\varphi_A(b)$. We say a semi-module $A$ is of Hodge type $\mu$ if $(\mu_\t(i))_{1 \le i \le n}$ is a permutation of $(\varphi_A(b))_{b \in \bar A_{\t+1}}$ for $\t \in \ZZ_d$.

Let $\L=(\L_\t)_{\t \in \ZZ_d} \in H(L)/K_H$. We set $$A(\L)=\{\h(v); 0 \neq v \in \L_\t \text{ for some } \t \in \ZZ_d\} \subseteq O.$$ For $a \in A(\L)$ define $\varphi_\L(a)$ to be the maximal integer $l$ such that $t^{-l} \g\s(v) \in \L$ for some $v \in \L$ with $\h(v)=a$. We set $\bar A(\L) =\overline{A(\L)}$, $\bar A^k_\t(\L)=\bar A(\L)^k_\t$ and so on.
\begin{lem} \cite[Corollary 5.10]{H} \label{lattice}
We have $\L \in X_\mu(\g)$ if and only if $A(\L)$ is a semi-module of Hodge type $\mu$ and $\varphi_\L = \varphi_{A(\L)}$.
\end{lem}

\subsection{} \label{equation} Let $A$ be a semi-module of Hodge type $\mu$. Let $\iota \in \ZZ_d$. Set $Y_\iota=\{\max_{\le} \bar A^k_\iota; 1 \le k \le h\}$. Let\begin{align*}V(A) &=\{(b, j) \in \bar A \times \ZZ_{\ge 1}; b+j \in \bar A, \ \varphi_A(b) > \varphi_A(b+j)\}; \\ W(A, \iota) &= \{(b, j) \in Y_\iota \times \ZZ_{\ge 1}; b + j \notin A\}; \\ D(A, \iota) &=V(A) \cup W(A, \iota). \end{align*} Notice that $V(A) \cap W(A, \iota) = \emptyset$. Let $\preceq_\iota$ be a partial order on $\bar A$ such that $r_A\i(b) \preceq_\iota b$ for $b \in \bar A \setminus Y_\iota$. This induces a partial order on $\bar A \times \ZZ_{\ge 0}$, which is still denoted by $\preceq_\iota$, such that $(b, j) \preceq_\iota (b', j')$ if either $j < j'$ or $j=j'$ and $b \preceq_\iota b'$.

\

For $x=(x_{b, j})_{(b, j) \in D(A, \iota)} \in \AA^{D(A, \iota)}$ we consider the following equations:

(0) $v(a) \in \sum_{j=0}^\infty \a_{a, j} e_{a+j}$ with coefficients $\a_{a, j} \in \kk$ such that $\a_{a, 0}=1$;

(1) $v(b)=e_b + \sum_{(b, j) \in W(A, \iota)} x_{b, j} e_{b+j} + \sum_{(b, j) \in V(A)} x_{b, j} v(b+j)$ for $b \in Y_\iota$;

(2) $v(b)=t^{-\varphi_A(r_A\i(b))} \g\s(v(r_A\i(b))) + \sum_{(b, j) \in V(A)} x_{b, j} v(b+j)$ for $b \in \bar A \setminus Y_\iota$;

(3) $v(a)=t v(a-n)$ if $a \in n + A$;

\

Thanks to Lemma \ref{exist} below, such vectors $v(a)$ for $a \in A$ always exist and are unique. We set $\L(x)=(\L_\t(x))_{\t \in \ZZ_d}$, where $\L_\t(x)$ is the $\co$-lattice spanned by $v(b)$ for $b \in \bar A_\t$. We say $(v(a))_{a \in A}$ is the normalized basis for $\L(x)$ or $x \in \AA^{D(A, \iota)}$. Moreover, we denote by $\cl(A, \iota)$ the set of points $x \in \AA^{D(A, \iota)}$ such that

(4) $t^{-\varphi_A(b)} \g\s(v(b)) \in \L(x)$ for $b \in \bar A$.

Thanks to (2), the condition (4) is equivalent to

(4') $t^{-\varphi_A(r_A\i(b))} \g\s(v(r_A\i(b))) \in \L(x)$ for $b \in Y_\iota$.

\begin{lem} [{\rm cf. Claim 1 of \cite[Theorem 4.3]{V}}] \label{exist}
For each point $x \in \AA^{D(A, \iota)}$, there exists a unique collection of vectors $(v(a))_{a \in A}$ for which the equations (0)-(3) hold. Moreover, the coefficients $\a_{b, j}$, viewed as functions on $\AA^{D(A, \iota)}$, belong to the polynomial ring $P_{b, j}:=\kk[X_{c, i}; (c, i) \in D(A, \iota), (c, i) \preceq_\iota (b, j)]$.
\end{lem}
\begin{proof}
By induction on the partial order $\preceq_\iota$ on $\bar A \times \ZZ_{\ge 0}$, we show that there exist unique coefficients $\a_{b, j} \in P_{b, j}$ satisfying (0)-(2), modulo the lattice $\sum_{l=j+1}^\infty \kk e_{b+l}$.

If $j=0$, then $\a_{b, j} \equiv 1 \in \kk=P_{b, j}$. Suppose $j \ge 1$ and the statement holds for all pairs $(b', j')$ such that $(b', j') \prec_\iota (b, j)$. We show it also holds for $(b, j)$. By induction hypothesis, the equations (1) and (2) holds modulo $\sum_{l=j+1}^\infty \kk e_{b+l}$ if and only if \begin{align*} \tag{*} \a_{b, j} = \begin{cases} X_{b, j} + \sum_{1 \le i \le j-1, \ (b, i) \in V(A)} X_{b, i} \a_{b+i, j-i}, &\text{ if } b \in Y_\iota, (b, j) \in W(A, \iota); \\ \sum_{1 \le i \le j, \ (b, i) \in V(A)} X_{b, i} \a_{b+i, j-i}, &\text{ if } b \in Y_\iota, (b, j) \notin W(A, \iota); \\ \a_{r_A\i(b), j}^q + \sum_{1 \le i \le j, \ (b, i) \in V(A)} X_{b, i} \a_{b+i, j-i}, &\text{ otherwise. } \end{cases}\end{align*} As $(b, i) \preceq_\iota (b, j)$, $(b+i, j-i) \prec_\iota (b, j)$ for $1 \le i \le j$ and $(r_A\i(b), j) \prec_\iota (b, j)$ for $b \in \bar A \setminus Y_\iota$, the coefficients $\a_{b+i, j-i}, \a_{r_A\i(b), j} \in P_{b, j}$ is uniquely determined by induction hypothesis. So $\a_{b, j} \in P_{b, j}$ is also uniquely determined.
\end{proof}

\begin{lem} [{\rm cf. Claim 2 \& 3 of \cite[Theorem 4.3]{V}}]\label{varphi}
For $x \in \AA^{D(A, \iota)}$ we have $A(\L(x)) = A$. Moreover, $\L(x) \in X_\mu(\g)$ if $x \in \cl(A, \iota)$.
\end{lem}
\begin{proof}
The equality $A(\L(x)) = A$ follows from the observation that $\bar A \subseteq A(\L(x))$ and that $b \notin b' + n\ZZ$ if $b \neq b' \in \bar A$.

Assume $x \in \cl(A, \iota)$. By (2) and (4) we have $\varphi_{\L(x)}=\varphi_A$. So $\L(x) \in X_\mu(\g)$ as $A$ is a semi-module of Hodge type $\mu$ (see Lemma \ref{lattice}).
\end{proof}

\begin{lem} \label{finite}
The natural projection $\cl(A, \iota) \to \AA^{V(A)}$ is finite and smooth.
\end{lem}
\begin{proof}
Set $\L=\L(x)$. Let $b \in Y_\iota$ and $b'=r_A\i(b)$. Write $$t^{-\varphi_A(b')} \g\s(v(b'))=v(b) + \sum_{j=1}^\infty \b_{b, j} v(b + j),$$ where $\b_{b, j} \in \kk[\AA^{D(A, \iota)}]$ and $v(b + j)=e_{b + j}$ if $b + j \notin A$, that is, $(b, j) \in W(A, \iota)$. By definition, $\cl(A, \iota) \subseteq \AA^{D(A, \iota)}$ is the zero locus of the coefficients $\b_{b, j}$ such that $(b, j) \in W(A, \iota)$.

Let $s=n' d=|\bar A^k|$ for any $k \in \ZZ$. By Lemma \ref{exist} (*) and \S\ref{equation} (3) we have

(i) $\b_{b, j} \in \a_{b', j}^q - \a_{b, j} + P_{j-1}$, where $P_{j-1}=\kk[X_{c, i}; (c, i) \in D(A, \iota), 1 \le i \le j-1]$.

(ii) $\a_{b', j} \in \a_{b, j}^{q^{s-1}} + P_{j-1}[X_{c, j}; (c, j) \in V(A), b \prec_\iota c]$.

(iii) $\a_{b, j} \in X_{b, j} + P_{j-1}$ if $(b, j) \in W(A, \iota)$.

Therefore, the coefficients $\b_{b, j}$ for $(b, j) \in W(A, \iota)$ are of the form $$\b_{b, j}=X_{b, j}^{q^s} - X_{b, j} + \d_{b, j}$$ for some $\d_{b, j} \in P_{j-1}[X_{c, j}; (c, j) \in V(A), b \prec_\iota c]$. Using the partial order $\preceq_\iota$ on $W(A, \iota)$, we see that the the Jacobian matrix $(\frac{\partial \b_{b, j}}{\partial X_{b', j'}})_{(b, j), (b', j') \in W(A, \iota)}$ is invertible. So the projection $\cl(A, \iota) \to \AA^{V(A)}$ is finite and smooth.
\end{proof}

\begin{lem} [{\rm cf. Claim 4 of \cite[Theorem 4.3]{V}}] \label{unique}
Let $\L \in X_\mu(\g)$ and $\iota \in \ZZ_d$. Then there exists a unique point $x \in \cl(A(\L), \iota)$ such that $\L(x)=\L$.
\end{lem}
\begin{proof}
Set $A=A(\L)$. For each $k \in \ZZ_{\ge 0}$ we define a collection $(v_k(a))_{a \in A}$ of vectors in $\L$ and parameters $x_{b, j} \in \kk$ with $(b, j) \in D(A, \iota)$ and $1 \le j \le k$ such that

(i)  $\h(v_k(b))=b$ and $t^{-\varphi_A(b)} \g\s (v_k(b)) \in \L$ for $b \in \bar A$;

(ii) the equations (0)-(3) hold for $v_k(b)$ modulo $\sum_{l=k+1}^\infty \kk e_{b+l}$;

(iii) $v_k(b)-v_{k-1}(b) \in \sum_{l=k}^\infty \kk e_{b+l}$ for $b \in \bar A$;

(iv) $v_k(a)=t v_k(a-n)$ for $a \in n + A$.

If $k=0$, as $\varphi_A=\varphi_\L$ (see Lemma \ref{lattice}) we can take $v_0(b) \in \L$ for $b \in \bar A$ such that (i)-(iv) hold. Assume $(v_{k-1}(a))_{a \in A}$ and $ x_{b, j}$ for $1 \le j \le k-1$ are already constructed. Then by induction on the partial order $\preceq_\iota$ on $\bar A$, one can construct $(v_k(a))_{a \in A}$ and $x_{b, k}$ such that (i)-(iv) hold. By (iii) and (iv) the limit $v(a)=\lim_{k \to \infty} v_k(a) \in \L$ exits for $a \in A$. Since $A(\L(x))=A$ and $\L(x) \subseteq \L$, we have $\L(x)=\L$. Moreover, by (i) and (ii), $(v(a))_{a \in A}$ satisfies (0)-(4) for $x$. So $x \in \cl(A, \iota)$ as desired.

Let $x' \in \cl(A, \iota)$ such that $\L(x')=\L$. Let $(v'(a))_{a \in A}$ be the normalized basis for $x'$. Suppose $x' \neq x$. Let $$(b, j) \in \min_{\preceq_\iota} \{(c, i) \in D(A, \iota); x_{c, i} \neq x_{c, i}'\}.$$ By Lemma \ref{exist}, we have $\a_{c, i}' = \a_{c, i}$ if $(c, i) \prec_\iota (b, j)$. Thus Lemma \ref{exist} (*) tells that $\h(v(b)-v'(b))=b+j$. On the other hand, as $v(b)-v'(b) \in \L$ and $t^{-\varphi_A(b)} \g\s(v(b)-v'(b)) \in \L$, we deduce that $b+j \in A$ and $\varphi_A(b+j) \ge \varphi_A(b)$. This contradicts the fact that $(b, j) \in D(A, \iota)$.
\end{proof}

Let $\ca_\mu$ be the set of semi-modules of Hodge type $\mu$, and let $\ca_\mu^\tp$ be the set of semi-modules $A \in \ca$ such that $\dim \cl(A)=|V(A)|=\dim X_\mu(\g)$. Here $\cl(A)=\cl(A, \iota)$ for some/any $\iota \in \ZZ_h$.
\begin{cor} \label{dec}
We have the following decompositions \begin{align*} X_\mu(\g) &= \sqcup_{A \in \ca_\mu} \cl(A) \\  \Irr X_\mu(\g) &= \sqcup_{A \in \ca_\mu^\tp} \Irr \cl(A), \end{align*} where each $\cl(A)$ is a locally closed subvariety of $ X_\mu(\g)$.
\end{cor}
\begin{proof}
The first decomposition follows from Lemma \ref{lattice}, Lemma \ref{varphi} and Lemma \ref{unique}. The second decomposition follows from the first one and the fact that each locally closed subvariety, which is a priori bounded and of finite type, intersects only finitely many strata in the decomposition. The last claim follows from Lemma \ref{exist}.
\end{proof}

\subsection{}\label{omega}
For $1 \le k \le h$ define $\o_k \in \JJ_\g=\JJ_\g^H$ such that \begin{align*} \o_k(e_a) = \begin{cases} e_{a+h}, & \text{ if }  a \in O^k; \\ e_a, & \text{ otherwise.} \end{cases} \end{align*} We denote by $\Omega_\g \subseteq \JJ_\g$ the subgroup generated by $\o_k$ for $1 \le k \le h$. Then $\Omega_\g$ is a free abelian group of rank $h$.

For $X, X' \subseteq O$, we write $X \le X'$ if $X_\t \le X_\t'$ for each $\t \in \ZZ_d$. We say a semi-module $A$ is {\it ordered} if $\bar A^1 \le \bar A^2 \le \cdots \le \bar A^h$.

Let $\bc \in \Irr X_\mu(\g)$. Denote by $A(\bc) \in \ca_\mu^\tp$ the unique semi-module such that $\cl(A(\bc))$ contains an open dense subset of $\bc$. For $1 \le i, j \le h$ we write $\bar A^i(\bc) \ll \bar A^j(\bc)$ if there exists $\o \in \Omega_\g$ such that $\bar A(\o \D^i) = A(\o \D^i) \setminus (n + A(\o \D^i)) \le \bar A^j(\bc)$ and $\o \D^i \subseteq \L$ for $\L \in \bc$. Here $\D^i=(\D_\t^i)_{\t \in \ZZ_d}$ with $\D_\t^i=\oplus_{k=0}^\infty e_{\t, i+k h}$. We have $\bar A^i(\bc) \le \bar A^j(\bc)$ if $\bar A^i(\bc) \ll \bar A^j(\bc)$.
\begin{lem} \label{order}
Let $\o=\prod_k \o_k^{p_k} \in \Omega_\g$ with $p_k \in \ZZ_{\ge 0}$. Let $1 \le i \neq j \le h$ such that $1 \le p_j=\max\{p_k; 1 \le k \le h\}$. Then for $\bc \in \Irr X_\mu(\g)$ we have

(1) $\bar A^i(\o^l \bc) \ll \bar A^j(\o^l \bc)$ if $p_i=0$ and $l \gg 0$;

(2) $\bar A^i(\o \bc) \ll \bar A^j(\o \bc)$ if $\bar A^i(\bc) \ll \bar A^j(\bc)$.
\end{lem}
\begin{proof}
As $p_j=\max\{p_k; 1 \le k \le h\}$, we have $A^j(\o^l \L) \subseteq l h p_j + A^j(\L)$ for $l \in \ZZ_{\ge 0}$.

(1) By \cite[Lemma 6.1]{H}, there exists $r \gg 0$ such that $\o_i^r \D^i \subseteq \L$ for $\L \in \cl(A(\bc))$. Since $p_i=0$, $\o^l \o_i^r \D^i = \o_i^r \D^i$. Thus, if $l \gg 0$, we have $\o^l \o_i^r \D^i \subseteq \o^l \L$ and $\bar A^i(\o_i^r \D^i) \le l h p_j + \bar A^j(\L)$. So $\bar A^i(\o_i^r \D^i) \le \bar A^j(\o^l \L)$ since $A^j(\o^l \L) \subseteq l h p_j + A^j(\L)$. This means $A^i(\o^l \bc) \ll A^j(\o^l \bc)$ as desired.

(2) Let $\o' \in \Omega_\g$ such that  $\bar A(\o' \D^i) \le \bar A^j(\bc)$ and $\o' \D^i \subseteq \L$ for $\L \in \bc$. Notice that $\bar A(\o \o' \D^i)=h p_i + \bar A(\o' \D^i)$. So $\o \o' \D^i \subseteq \o \L$ and $$\bar A(\o \o' \D^i) = h p_i + \bar A(\o' \D^i) \le p_j h + \bar A^j(\L),$$ which means $\bar A(\o \o' \D^i) \le \bar A^j(\o\L)$ and hence $A^i(\o \bc) \ll A^j(\o \bc)$ as desired.
\end{proof}

\begin{cor}
Let $\bc \in \Irr X_\mu(\g)$. Then there exist $j \in \Omega_\g$ and an ordered semi-module $A \in \ca_\mu^\tp$ such that $j \bc \in \Irr \cl(A)$.
\end{cor}
\begin{proof}
We argue by induction on $k$ that there exist $\o \in \Omega_\g$ such that $\bar A^i(\o \bc) \ll \bar A^k(\o \bc) \ll \cdots \ll \bar A^h(\o \bc)$ for $1 \le i \le k-1$. If $k=h+1$, there is nothing to prove. Suppose the statement is true for $k=j+1$ for some $1 \le j \le h$. We show it is also true for $k=j$. By induction hypothesis, there exists $\o' \in \Omega_\g$ such that $\bar A^i(\o' \bc) \ll \bar A^{j+1}(\o' \bc) \ll \cdots \ll \bar A^h(\o' \bc)$ for $1 \le i \le j$. Let $\o''=\o_j \cdots \o_h \in \Omega_\g$. Then Lemma \ref{order} tells that for $l \gg 0$ one has $\bar A^i((\o'')^l \o' \bc) \ll \bar A^j((\o'')^l \o' \bc) \ll \cdots \ll \bar A^h((\o'')^l \o' \bc)$ for $1 \le i \le j-1$. So the induction is finished.
\end{proof}

\subsection{} \label{rigid} Let $A$ be an ordered semi-module of Hodge type $\mu$. We fix $\iota \in \ZZ_d$.
\begin{lem} \label{sepa}
Let $x \in \cl(A, \iota)$ and let $(v(a))_{a \in A}$ be the corresponding normalised basis. For $1 \le k \le h$ and $a \in A^k$ we have $$v(a) \in  \sum_{i=0}^{h-k} \sum_{j=0}^\infty \kk e_{a+i+jh}.$$
\end{lem}
\begin{proof}
The statement follows from the following two facts:

(1) $A$ is ordered and hence $\bar A^1 \le \bar A^2 \le \cdots \le \bar A^h$;

(2) for $1 \le i \le h$ and $b \in \bar A^i$ we have $(b, j) \notin D(A, \iota)$ if $\bar A^i > \bar A^{i+j}$.
\end{proof}

For $1 \le k \le h$ we set $y_\iota^k=\max \bar A_\iota^k$ and $s=n'd=|\bar A^k|$. For $\xi=(\xi_i)_{1 \le i \le h-k} \in \FF_{q^s}^{h-k}$, we define $n_{\iota, k, \xi} \in \JJ_\g$ by $e_a \mapsto e_a$ if $a \notin O^k$ and $e_a \mapsto  e_a + \sum_{i=1}^{h-k} \xi_i^{q^l} e_{a+i}$ if $a \in f^l(y_\iota^k)+n\ZZ$ for some $l \in \ZZ$.
\begin{lem} \label{piece}
If $\bar A^{k+1} \ge h + \bar A^k$, then $(y_\iota^k, i) \in W(A, \iota)$ for $1 \le i \le h-k$ and $x_{y_\iota^k, i} \in \FF_{q^s}$ for each $x \in \cl(A, \iota)$. As a consequence, $$\cl(A, \iota) = \sqcup_{\xi\in \FF_{q^s}^{h-k}} \cl(A, \iota, k, \xi),$$ where $\cl(A, \iota, k, \xi)=\{x \in \cl(A, \iota); x_{y_\iota^k, i}=\xi_i \ {\rm for } \ 1 \le i \le h-k\}$. Moreover, $$n_{\iota, k, \xi'} \cl(A, \iota, k, \xi) = \cl(A, \iota, k, \xi+\xi').$$
\end{lem}
\begin{proof}
As $A$ is ordered and $\bar A^{k+1} \ge h + \bar A^k$, we have $b+i \notin A$ for $b \in \bar A^k$ and $1 \le i \le h-k$. So $(y_\iota^k, i) \in W(A, \iota)$ if $b=y_\iota^k$ and $(b, i) \notin D(A, \iota)$ if $b \in \bar A^k \setminus \{y_\iota^k\}$. Then Lemma \ref{exist} (*) tells that $\a_{b, i}=X_{y_\iota^k, i}^{q^l}$ if $b \in f^l(y_\iota^k) + n \ZZ$ for some $0 \le l \le s-1$. In view of Lemma \ref{sepa} and the requirement \S\ref{equation} (4'), we deduce that $x_{y_\iota^k, i} \in \FF_{q^s}$ as desired.
\end{proof}

\begin{lem} \label{tight}
Let $1 \le  k \le h-1$ such that $\bar A^{k+1} \ge h + \bar A^k$. For $\L \subseteq \cl(A, \iota, k, 0)$ we have $\bar A^l(\o_k \L)=h + \bar A^l(\L)$ if $\bar A^l=\bar A^k$ and $\bar A^l(\o_k \L)=\bar A^l(\L)$ otherwise. Here $\o_k \in \Omega_\g$ is defined in \S\ref{omega}.
\end{lem}
\begin{proof}
Set $A=A(\L)$. Let $(v(a))_{a \in A}$ be the normalised basis for $\L \in \cl(A, \iota, k, 0)$. By Lemma \ref{sepa} and Lemma \ref{piece} we have \begin{align*} v(a) \in  e_a + \sum_{i=0}^{h-k} \sum_{j=1}^\infty \kk e_{a+i+jh}  \text { for } a \in A^k. \end{align*} By computing $\h(v(a))$ for $a \in A$, we have $h + A^k \subseteq A^k(\o_k \L)$ and $A^l \subseteq A^l(\o_k \L)$ if $A^l \neq A^k$. On the other hand, let $v \in \L$ such that $v = v(a) + \sum_{j=1}^\infty \b_j v(a+j)$ for some $\b_j \in \kk$. We can assume $\b_j=0$ if $a+j \notin A$. If $a \in A^l \neq A^k$, then $\h(\o_k(v))=\h(v)=a \in A^l$. Assume $a \in A^k$. Then $\h(\o_k(v))=j_0+a$ for some $1 \le j_0 \le h$. If $j_0=h$, then $\h(\o_k(v))=h+a \in h + A^k$. If $h-k+1 \le j_0 \le h-1$, then $\bar A^k \ge \bar A^{k+j_0}$ and hence $\h(\o_k(v))=a+j_0 \in A^{k + j_0}$. If $1 \le j_0 \le h-k$, then $\b_{j_0} \neq 0$ (as $\a_{a, i}=0$ for $1 \le i \le h-k$) and hence $\h(\o_k(v))=a+j_0 \in A^{k+ j_0}$. Therefore, we have $A^k(\o_k \L)= h + A^k$ and $A^l(\o_k \L) = A^l$ if $A^l \neq A^k$.
\end{proof}

We say $A$ is {\it rigid} if $\bar A^k \le \bar A^{k+1}$ and $\bar A^k + h \nleqslant \bar A^{k+1}$ for each $1 \le k \le h-1$. By Lemma \ref{order}, Lemma \ref{piece} and Lemma \ref{tight} we deduce that
\begin{cor} \label{to-rig}
Let $\bc \in \Irr X_\mu(\g)$. Then there exist $j \in \JJ_\g$ and a rigid semi-module $A$ such that $j \bc \in \Irr \cl(A)$.
\end{cor}

\

For $1 \le k \le h-1$ we define $s_k \in \JJ_\g$ by \begin{align*} s_k(e_a) = \begin{cases} e_{a-1}, & \text{ if }   a \in O^{k+1}; \\ e_{a+1}, & \text{ if } a \in O^k; \\ e_a, & \text{ otherwise.}  \end{cases}  \end{align*} We denote by $\fs_\g \subseteq \JJ_\g$ the subgroup generated by $s_k$ for $1 \le k \le h-1$. Then $\fs_\g$ is isomorphic to the symmetry group of $h$ letters.

\begin{lem} \label{fix}
If $A$ is rigid, then $\fs_\g$ preserves $\Irr \cl(A)$.
\end{lem}
\begin{proof}
It suffices to show $s_k$ preserves $\Irr \cl(A)$ for each $1 \le k \le h-1$.

\

Case(1): $\bar A^{k+1} = 1 + \bar A^k$.

We claim that $s_k \cl(A)=\cl(A)$. Indeed, let $\t \in \ZZ_d$ and $x \in \cl(A, \t)$. Let $v(a)_{a \in A}$ be the normalized basis for $x$. Since $\bar A^k \le \bar A^{k+1} = 1 + \bar A^k$, we have $h=n$ and $\varphi_A(b+1)=\varphi_A(b)$ for $b \in \bar A^k$. Thus $(\bar A^k, 1) \cap D(A, \t) = \emptyset$ and $v(a) \in e_a + \sum_{j=2}^\infty \kk e_{a+j}$ for $a \in A^k$. In particular, \begin{align*}\h(s_k(v(a)))=\begin{cases} a+1, & \text{ if } a \in A^k; \\ a-1, & \text{ if } a \in A^{k+1}; \\ a, & \text{ otherwise.}  \end{cases}\end{align*} Therefore, $A=A(\L(x)) \subseteq A(s_k \L(x))$. On the other hand, let $v \in \L(x)$ and $a=\h(v) \in A$. If $a \in A \setminus (A^k \cup A^{k+1})$, then $\h(s_k(v))=a \in A$. Otherwise, $\h(s_k(v))=a-1 \in A^k$ if $a \in A^{k+1}$ and $\h(s_k(v)) \in \{a, a+1\} \subseteq A^k \cup A^{k+1}$ if $a \in A^k$. Therefore, $A(s_k \L(x))=A$ and the claim is proved.

\

Case(2): $\bar A^{k+1} \neq 1 + \bar A^k$.

First we claim that

(i) there exist $\iota_0 \in \ZZ_d$ and $y \in \bar A_{\iota_0}^k$ such that $y+1 \in \bar A_{\iota_0}^{k+1}$ and $(y, 1) \in V(A)$. In particular, $y=\max \bar A_{\iota_0}^k$ as $\bar A_{\iota_0}^{k+1} \ge \bar A_{\iota_0}^k$.

Indeed, since $A$ is rigid, there exists $b \in \bar A^k$ with $b+1 \in \bar A^{k+1}$. Assume (i) fails. Then $(b, 1) \notin V(A)$, that is, $\varphi_A(b) \le \varphi_A(b+1)$. If $\varphi_A(b) < \varphi_A(b+1)$, then $r_A(b) > r_A(b+1)$, contradicting that $\bar A^{k+1} \ge \bar A^k$. So we have $\varphi_A(b) = \varphi_A(b+1)$ and hence $r_A(b)+1=r_A(b+1)$. Repeating this argument, we deduce that $r_A^l(b)+1=r_A^l(b+1)$ and $\varphi_A(r_A^l(b))=\varphi_A(r_A^l(b+1))$ for $l \in \ZZ$. This means $\bar A^{k+1} = 1 + \bar A^k$, contradicting our assumption. So (i) is proved.

By (i) we have $(y, 1) \in V(A)$ and hence all the coefficients $\a_{b, 1}$ for $b \in \bar A^k$ are non-zero polynomials in $\kk[X_{a, 1}; (a, 1) \in V(A), a \in A^k]$. Let $$U=\{x \in \AA^{V(A)}; \a_{b, 1}(x) \neq 0 \text{ for } b \in \bar A^k \},$$ which is an open dense subset of $\AA^{V(A)}$. Let $U' \subseteq \cl(A, \iota_0)$ be the preimage of $U$ under the natural projection $\cl(A, \iota_0) \to \AA^{V(A)}$. By Lemma \ref{finite}, $U'$ is open dense in $\cl(A, \iota_0)=\cl(A)$. Let $x \in U'$ and let $(v(a)_{a \in A})$ the corresponding normalized basis. By definition, \begin{align*} v(a) \in  e_a + \kk^\times e_{a+1} + \sum_{k=2}^\infty \kk e_{a+k} \text{ for } a \in A^k.  \end{align*} In particular, $\h(s_k(v(a)))=a$ for $a \in A \setminus A^{k+1}$. Moreover, for each $a \in A^{k+1}$, there exists $c \in \bar A^k$ such that $a \equiv c+1 \mod n$. Moreover, since $\bar A^{k+1} \ge \bar A^k$, $a-(c+1)=i_0 n$ for some $i_0 \in \ZZ_{\ge 0}$. Then $$v(a)- t^{i_0} \a_{c, 1}\i v(c) \in \a_{c, 1}\i e_{a-1} + \sum_{k=1}^\infty \kk e_{a+k}.$$ So $\h(s_k( v(a)- t^{i_0} \a_{c, 1}\i v(c)))=a$. Therefore, $A \subseteq A(s_k \L(x))$. On the other hand, let $v \in \L(x)$ such that $v = v(a) + \sum_{j=1}^\infty \b_j v(a+j)$ for some $\b_j \in \kk$. We can assume $\b_j=0$ if $a+j \notin A$. If $a \in A \setminus (A^{k+1} \cup A^k)$, then $\h(s_k(v))=a \in A$. If $a \in A^{k+1}$, then $\h(s_k(v))=a-1 \in A^k$. If $a \in A^k$, then $\h(s_k(v))$ equals either $a \in A^k$ or $a+1$. In the latter case, we have $\b_1 \neq 0$ (as $\a_{a, 1} \neq 0$) and hence $a+1 \in A^{k+1}$. Therefore, $A(s_k \L(x))=A$ and hence $s_k (U') \subseteq \cl(A)$. The proof is finished.
\end{proof}

\section{Orbits of irreducible components} \label{irr}
Let notations be as in \S\ref{semi}. Let $M = \prod_{1 \le k \le h} M^k$ with each $M^k \cong \Res_{\FF_{q^d} / \FF_q} GL_{n'}$. Moreover precisely, $M^k(L) = \prod_{\t \in \ZZ_d} GL(N_\t^k)$, where $N_\t^k=\oplus_{j=1}^{n'} L e_{\t, k+j h}$. Then $M \supseteq T_H$ is a semi-standard Levi subgroup of $H$. Notice that $\g$ is superbasic in $M(L)$.

\subsection{} We say $C=(C^k)_{1 \le k \le h}$ is a semi-module for $M$ if $C^k \subseteq O^k$ is bounded below, $n + C^k, f(C^k) \subseteq C^k$ and $n\ZZ + C^k = O^k$ for $1 \le k \le h$. Set $\bar C = C \setminus (n + C)$ and define $\varphi_C: C \to \ZZ$ and $r_C: \bar C \overset \sim \to \bar C$ analogously as in \S \ref{setup}. Let $\l=(\l_\t^k) \in X_*(T_M)=X_*(T_H)$ be an $M$-dominant cocharacter (with respect to the Borel subgroup $B_M= B_H \cap M$). We say a semi-module $C$ is of Hodge type $\l$ if $(\l_\t^k(k+ih))_{0 \le i \le n'-1}$ is a permutation of $(\varphi_C(b))_{b \in \bar C_{\t+1}^k}$ for $1 \le k \le h$ and $\t \in \ZZ_d$.

Let $I_{\mu, \g} = I_{\mu, \g, M}$ be the set of $M$-dominant cocharacters $\l$ such that $\l$ is conjugate to $\mu$ under $W_H$ and $X_\l^M(\g) \neq \emptyset$. As $\mu$ is minuscule, we have $\l \in I_{\mu, \g}$ if and only if $\l \leq \mu$ is $M$-dominant and $m'=\sum_{\t \in \ZZ_d} \sum_{l=1}^{n'} \l_\t^k(k + l h)$ for $1 \le k \le h$.

Let $A$ be a semi-module for $H$ of Hodge type $\mu$. Let ${}^\g A:=(A^1, \dots, A^k)$ and let $\l_A$ be the $M$-dominant cocharacter such that $(\l_A)^k_\t(k+ih)_{0 \le i \le n'-1}$ is a permutation of $(\varphi_A(b))_{b \in \bar A_{\t+1}^k}$ for $1 \le k \le h$ and $\t \in \ZZ_d$. Then ${}^\g A$ is a semi-module for $M$ of Hodge type $\l_A \in I_{\mu, \g}$.

Suppose $A$ is ordered. Let $\iota \in \ZZ_d$ and $x \in \cl(A, \iota)$. Let $(v(a))_{a \in A}$ denote the normalized basis for $x$. For $1 \le k \le h$ and $a \in A^k$ we set $v^k(a)=\sum_{i=0}^\infty \a_{a, ih} e_{a+ih}$. Since $A$ is ordered, Lemma \ref{sepa} tells that $(v^k(a))_{a \in A^k}$ is the normalized basis for $x^k \in \cl(A^k, \iota)$, where $x^k_{b, j}=x_{b, j}$ for $(b, j) \in D(A^k, \iota)=V(A^k):=\{(b, j) \in V(A); b, b+j \in A^k\}$. The map $x \mapsto (x^k)_{1 \le k \le h}$ gives a morphism $$\b_{A, \iota}: \cl(A)=\cl(A, \iota) \to \cl({}^\g A) = (\cl(A^1), \dots, \cl(A^h)) \subseteq X_{\l_A}^M(\g).$$ On the other hand, let $N \subseteq H$ be the unipotent subgroup such that $$N(L) e_{\t, i} = e_{\t, i} + \sum_{l=i+1}^h \sum_{k=0}^{n'-1} L e_{\t, l + k h}$$ for $1 \le i \le h$. Then $P := M N =N M \subseteq H$ is a semi-standard parabolic subgroup. The Iwasawa decomposition gives a natural projection $$\b_\g: H(L) / K = P(L) K / K  \to M(L) / M(\co).$$ Observing that $v(a) \in P(L) e_a$ for $a \in A$, we have
\begin{lem}
If $A$ is ordered, then $\b_{A, \iota}=\b_\g |_{\cl(A, \iota)} = \b_\g |_{\cl(A)}$. In particular, the morphism $\b_{A, \iota}$ is independent of the choice of $\iota \in \ZZ_h$.
\end{lem}

\subsection{} Let $\car^{\tp}_\mu$ be the set of rigid semi-modules $A$ of Hodge type $\mu$ such that $\dim \cl(A) = \dim X_\mu(\g)$. For $A', A'' \in \car^\tp$ we write $A' \sim A''$ if $A'=A'' + k h$ for some $k \in \ZZ$. We set $\tilde \car^{\tp}_\mu = \car_\mu^\tp / \sim$.

Analogously, for each $M$-dominant cocharacter $\l \leq \mu$ we denote by $\cc_\l^\tp$ the set of semi-modules $C$ for $M$ of Hodge type $\l$ such that $\dim \cl(C) =\dim X_\l^M(\g)$. Then $\cc_\l^\tp$ admits an action by $\Omega_\g$ (see \S \ref{omega}) such that $$(\o_1^{p_1} \cdots \o_h^{p_h}) (C^1, \dots, C^h)=(p_1 + C^1, \dots, p_h+ C^h).$$ Actually, we have $\o C(\L) = C(\o \L)$ for $\L \in X_\l^M(\g)$ and $\o \in \Omega_\g$. Here $C(\L)$ denotes the semi-module for $M$ associated to $\L$. Let $\tilde \cc_\l^\tp$ be the set of $\Omega_\g$-orbits in $\cc_\l^\tp$.

\begin{lem} \label{top}
Let $A$ be a rigid semi-module of Hodge type $\mu$. Then $A \in \car^\tp_\mu$ if and only if ${}^\g A=(A^1, \dots, A^h) \in \cc_{\l_A}^\tp$.
\end{lem}
\begin{proof}
Let $m_{A, \t}^k=\sum_{b \in \bar A_{\t+1}^k} \varphi_A(b)$ for $\t \in \ZZ_d$ and $1 \le k \le h$. Since $A$ is a semi-module of Hodge type $\mu$, we have $m_\t=\sum_{k'} m_{A, \t}^{k'}$ and $m'=\sum_{\t'} m_{A, \t'}^k$ for $\t \in \ZZ_d$ and $1 \le k \le h$. For $1 \le i, j \le h$ set $$V_{i, j}(A)=\{(b, k) \in V(A); b \in \bar A^i, b+k \in \bar A^j\}.$$ Then $V(A)=\sqcup_{i, j} V_{i, j}(A)$ and $V_{i, j}(A) = \emptyset$ unless $i \le j$ since $A$ is ordered. As $\mu$ is minuscule, we can assume $\varphi_A(b) \in \{0, 1\}$ for $b \in \bar A$. In particular, $m_{A, \t}^k=\sharp \{b \in \bar A_{\t+1}^k; \varphi_A(b)=1\}$. Thus \begin{align*} \sum_{i=1}^h \sharp V_{i, i}(A) &= \sum_{i=1}^h \dim \cl(A^i) = \dim \cl({}^\g A) \\ &\le \dim X_{\l_A}^M (\g) \\ &= -\frac{1}{2}h(n'-1) + \frac{1}{2} \sum_{1 \le k \le h} \sum_{\t \in \ZZ_d} (n'-m_{A, \t}^k) m_{A, \t}^k \\ &=-\frac{1}{2}h(n'-1) + \frac{1}{2}n' \sum_k \sum_\t m_{A, \t}^k - \frac{1}{2} \sum_k \sum_\t (m_{A, \t}^k)^2 \\ &= -\frac{1}{2}h(n'-1) + \frac{1}{2}hn'm' - \frac{1}{2} \sum_k \sum_\t (m_{A, \t}^k)^2.\end{align*} Moreover, for $1 \le i < j \le h$ we have \begin{align*} \sharp V_{i, j}(A) &=\{(b, b') \in \bar A^i \times \bar A^j, \varphi_A(b)=0, \varphi_A(b')=1\} \\ &=\sum_{\t \in \ZZ_d} m^j_{A, \t} (n' - m^i_{A, \t}).\end{align*} Hence \begin{align*} \sum_{1 \le i < j \le h} \sharp V_{i, j}(A) &=n' \sum_{\t \in \ZZ_d} \sum_{1 \le k \le h} (k-1)m_{A, \t}^k - \sum_\t \sum_{1 \le i < j \le h} m_{A, \t}^i m_{A, \t}^j \\ &= \frac{1}{2}h(h-1)n'm' - \sum_\t \sum_{1 \le i < j \le h} m_{A, \t}^i m_{A, \t}^j.\end{align*} Thus \begin{align*} \dim \cl(A) &= \sharp V(A) = \sum_{1 \le i \le j \le h} \sharp V_{i, j}(A)\\ &= \dim \cl({}^\g A) - \dim X_{\l_A}^M(\g) -\frac{1}{2}h(n'-1)  \\ & \quad\ + \frac{1}{2}h^2n'm' - \frac{1}{2}\sum_\t (\sum_{k=1}^h m_{A, \t}^k)^2  \\ &= \dim \cl({}^\g A) - \dim X_{\l_A}^M(\g) -\frac{1}{2}h(n'-1)  \\ & \quad\ + \frac{1}{2}h^2n'm' - \frac{1}{2} \sum_\t (m_\t)^2.\end{align*} On the other hand, \begin{align*} \dim X_\mu(\g) &=-\frac{1}{2}(n-h) + \frac{1}{2}\sum_\t (n-m_\t)m_\t \\ &= -\frac{1}{2}(n-h)+ \frac{1}{2} n\sum_\t m_\t - \frac{1}{2}\sum_\t (m_\t)^2 \\ &=-\frac{1}{2}(n-h)+ \frac{1}{2} n m - \frac{1}{2}\sum_\t (m_\t)^2  \\ &= \dim \cl(A) + \dim X_{\l_A}^M(\g)- \dim \cl({}^\g A) \\ &\ge \dim\cl(A). \end{align*} Therefore, the equality holds if and only if $\dim \cl({}^\g A) = \dim X_{\l_A}^M(\g)$, that is, ${}^\g A \in \cc_{\l_A}^\tp$.
\end{proof}

\begin{lem} \label{N-conj}
If $A \in \car_\mu^\tp$, then all the irreducible components of $\cl(A)$ are conjugate by $N(L) \cap \JJ_\g$.
\end{lem}
\begin{proof}
Consider the surjective projection $\b_\g |_{\cl(A)}: \cl(A) \to \cl({}^\g A)$. By Lemma \ref{top}, $\cl({}^\g A) \cong \prod_k \AA^{|V(A^k)|}$ is an irreducible component of $X_{\l_A}^M(\g)$. Then \cite[Proposition 5.6]{HV} tells that all the irreducible components of $X_\mu(\g) \cap \b_\g\i(\cl({}^\g A))$ are conjugate by $N(L) \cap \JJ_\g$. So the statement follows from the inclusion $\Irr \cl(A) \subseteq \Irr ( X_\mu(\g) \cap \b_\g\i(\cl({}^\g A)))$.
\end{proof}

\begin{lem} \label{equiv}
Let $A, A' \in \car^\tp_\mu$. If $\JJ_\g \Irr \cl(A) = \JJ_\g \Irr \cl(A') \neq \emptyset$, then $A \sim A'$.
\end{lem}
\begin{proof}
Let $I \subseteq H(L)$ be the (standard) Iwahori subgroup such that $$I(e_a) \subseteq \kk^\times e_a + \sum_{j=1}^\infty \kk e_{a+j}.$$ By Bruhat decomposition we have $$\JJ_\g=(\JJ_\g \cap I) \Omega_\g \fs_\g (\JJ_\g \cap I) = (\JJ_\g \cap I) \fs_\g \Omega_\g (\JJ_\g \cap I),$$ where $\Omega_\g, \fs_\g \subseteq \JJ_\g$ are defined in \S\ref{omega} and \S\ref{tight}, respectively. By definition, $\JJ_\g \cap I$ preserves $\cl(A)$ and $\cl(A')$. Thanks to Lemma \ref{fix}, $\fs_\g$ preserves $\Irr \cl(A)$ and $\Irr \cl(A')$. Thus there exist $\L \in \cl(A)$ and $\o \in \Omega_\g$ such that $\o \L \in \cl(A')$. So $${}^\g(A')= C(\b_\g(\o \L))= C(\o \b_\g(\L)) = \o C(\b_\g(\L)) = \o ({}^\g A).$$ Thus $A \sim A'$ since $A$ and $A'$ are rigid.
\end{proof}

By the construction of $P=M N$, there exits $z=\s(z)$ in the Weyl group $W_H$ of $T_H$ in $H$ such that ${}^z P = {}^z M {}^z N$ is a standard parabolic subgroup. Moreover, we can assume further that ${}^z(B_H \cap M)=B_H \cap {}^z M$. In particular, $z(\l_M(\g))=\l_{{}^z M} ({}^z \g) = \l_H({}^z \g) =\l_H(\g)$ and ${}^z I_{\mu, \g, M}=I_{\mu, {}^z \g, {}^z M}$.
\begin{cor} \label{main''}
We have natural bijections $$\JJ_\g \backslash \Irr X_\mu(\g) \overset \sim \longrightarrow  \tilde \car_\mu^\tp(G)  \overset \sim \longrightarrow  \sqcup_{\l \in I_{\mu, \g}} \tilde \cc_\l^{\tp}.$$ In particular, $$|\JJ_\g \backslash \Irr X_\mu(\g)|=\dim V_\mu^{\hat H}(\l_H(\g)).$$
\end{cor}
\begin{proof}
The first bijection is due to Corollary \ref{to-rig} and Lemma \ref{N-conj} and Lemma \ref{equiv}. The second bijection is induced by the map $A \mapsto {}^\g A$ (see Lemma \ref{top}). As $\g$ is superbasic in $M(L)$, \cite[Thoerem 1.5]{HV} tells that $|\tilde \cc_\l^\tp|=\dim V_\l^{\hat M}(\l_M(\g))$. Thus $$|\JJ_\g \backslash \Irr X_\mu(\g)| = \sum_{\l \in I_{\mu, \g}} |\tilde \cc_\l^\tp| = \sum_{\l \in I_{\mu, \g}} \dim V_\l^{\hat M}(\l_M(\g)) = \dim V_\mu^{\hat H}(\l_M(\g)),$$ where the last equality follows from a similar argument in Proposition \ref{minu}. Notice that $\l_H(\g)=z(\l_M(\g))$ for some element $z=\s(z) \in W_H$. Therefore, $\dim V_\mu^{\hat H}(\l_M(\g)) = \dim V_\mu^{\hat H}(\l_H(\g))$ and the last statement follows.

\end{proof}

\end{document}